\documentclass[a4paper]{amsart}
\usepackage{amsmath,amsthm,amssymb,hyperref,mathrsfs}
\usepackage{aliascnt}
\usepackage{lmodern}
\usepackage[T1]{fontenc}

\usepackage[textsize=footnotesize]{todonotes}

\usepackage{xcolor}
\definecolor{dblue}{rgb}{0,0,0.70}
\hypersetup{
	unicode=true,
	colorlinks=true,
	citecolor=dblue,
	linkcolor=dblue,
	anchorcolor=dblue
}

\makeatletter
\expandafter\g@addto@macro\csname th@plain\endcsname{%
		\thm@notefont{\bfseries}
	}%
\expandafter\g@addto@macro\csname th@remark\endcsname{%
		\thm@headfont{\bfseries}
	}%
\makeatother

\newtheorem{theorem}{Theorem}[section]
\newtheorem*{theorem*}{Theorem}

\newaliascnt{lemma}{theorem}
\newtheorem{lemma}[lemma]{Lemma}
\aliascntresetthe{lemma}
\newtheorem*{lemma*}{Lemma}

\newaliascnt{proposition}{theorem}
\newtheorem{proposition}[proposition]{Proposition}
\aliascntresetthe{proposition}

\newaliascnt{corollary}{theorem}
\newtheorem{corollary}[corollary]{Corollary}
\aliascntresetthe{corollary}

\theoremstyle{remark}

\newaliascnt{remark}{theorem}
\newtheorem{remark}[remark]{Remark}
\aliascntresetthe{remark}
\newaliascnt{question}{theorem}
\newtheorem{question}[question]{Question}
\aliascntresetthe{question}

\newtheorem*{question*}{Question}

\newaliascnt{definition}{theorem}
\newtheorem{definition}[definition]{Definition}
\aliascntresetthe{definition}

\newaliascnt{example}{theorem}

\aliascntresetthe{example}

\renewcommand{\restriction}{\mathbin\upharpoonright}

\newcommand{\axiom}[1]{\mathsf{#1}}
\newcommand{\ZFC}{\axiom{ZFC}}
\newcommand{\AC}{\axiom{AC}}

\newcommand{\ZF}{\axiom{ZF}}

\newcommand{\HOD}{\mathrm{HOD}}
\newcommand{\GCH}{\axiom{GCH}}

\newcommand{\HS}{\axiom{HS}}

\DeclareMathOperator{\cf}{cf}

\DeclareMathOperator{\sym}{sym}

\DeclareMathOperator{\id}{id}
\DeclareMathOperator{\aut}{Aut}

\DeclareMathOperator{\Ult}{Ult}

\newcommand{\forces}{\mathrel{\Vdash}}

\newcommand{\PP}{\mathbb{P}}
\newcommand{\power}{\mathcal{P}}

\newcommand{\QQ}{\mathbb{Q}}
\newcommand{\RR}{\mathbb{R}}

\newcommand{\cU}{\mathcal U}
\newcommand{\sF}{\mathscr F}

\newcommand{\sG}{\mathscr G}
\newcommand{\sH}{\mathscr H}

\newcommand{\cS}{\mathcal S}

\def\ssubset{\mathrel{\vtop{\ialign{##\crcr$\hfil\subset\hfil$\crcr\noalign{\kern1.5pt\nointerlineskip}$\hfil\widetilde{}\hfil$\crcr\noalign{\kern1.5pt}}}}\vspace*{-0.3\baselineskip}}

\def\undertilde#1{\mathord{\vtop{\ialign{##\crcr
$\hfil\displaystyle{#1}\hfil$\crcr\noalign{\kern1.5pt\nointerlineskip}
$\hfil\tilde{}\hfil$\crcr\noalign{\kern1.5pt}}}}}

\newcommand{\tup}[1]{\langle#1\rangle}

\author{Yair Hayut}
\author{Asaf Karagila}
\thanks{The first author was partially supported by Lise Meitner FWF grant~2650-N35 and the ISF grant no.~1967/21. The second author was supported by the Royal Society Newton International Fellowship, grant no.~NF170989 and by UKRI Future Leaders Fellowship [MR/T021705/2].}
\thanks{No data are associated with this article.}
\address[Yair Hayut]{Einstein Institute of Mathematics, The Hebrew University of Jerusalem, Givat Ram, Jerusalem, 9190401, Israel}
\email{yair.hayut@mail.huji.ac.il}
\address[Asaf Karagila]{School of Mathematics,
University of Leeds,
Leeds, LS2~9JT, UK
}
\email{karagila@math.huji.ac.il}
\date{2 February, 2026}
\subjclass[2020]{Primary 03E25; Secondary 03E55, 03E35}
\keywords{Axiom of Choice, measurable cardinals, symmetric extensions,  elementary embeddings, Silver criterion}

\title{Small measurable cardinals}

\begin{document}
\begin{abstract}We continue the work from \cite{HayutKaragila:Critical} and make a small---but significant---improvement to the definition of $j$-decomposable system. This provides us with a better lifting of elementary embeddings to symmetric extensions. In particular, this allows us to more easily lift weakly compact embeddings and thus preserve the notion of weakly critical cardinals. We use this improved lifting criterion to show that the first measurable cardinal can be the first weakly critical cardinal or the first Mahlo cardinal, both relative to the existence of a single measurable cardinal. However, if the first inaccessible cardinal is the first measurable cardinal, then in a suitable inner model it has Mitchell order of at least $2$.
\end{abstract}
\maketitle
\section{Introduction}
In this short note we improve the basic lifting theorem for elementary embeddings to symmetric extensions, originally stated in \cite{HayutKaragila:Critical}. This improvement allows us to lift weakly compact embeddings and therefore preserve weakly critical cardinals (which are a strong choiceless version of weakly compact cardinals). We apply this new method to answer a question of Itay Kaplan (Question~3.4 in \cite{HayutKaragila:Critical}) and show that the least measurable can be the least weakly critical cardinal. We also show that the same technique can---essentially---provide us with a model in which the first Mahlo cardinal is the first measurable cardinal, although no embeddings are involved in this case.

This flavour of ``identity crisis'' extends the well known result of Jech that $\omega_1$ can be measurable, \cite{Jech:1968} (and this is equiconsistent with a single measurable cardinal).\footnote{See, for example, Lemma~19.1 in \cite{Jech2003}.} However, this approach fails when attempting to construct a model where the least measurable cardinal is the least inaccessible cardinal. Indeed, we show that if $\kappa$ is the least inaccessible cardinal and the least measurable cardinal, then in an inner model we have that $o(\kappa)\geq 2$. This shows that that the consistency strength required is significantly higher for this result to be possible.

\subsection*{Acknowledgements}
The authors would like to thank Ralf Schindler for making helpful suggestions which improved \autoref{thm:failure-of-covering} to its current form. We also thank the anonymous referee for their very helpful suggestions.

\section{Technical preliminaries}\label{section:prelims}
Let us begin by clarifying some of the large cardinal notions often studied by set theorists in $\ZFC$. We will (re)introduce critical and weakly critical cardinals later.

We say that $\kappa$ is an inaccessible cardinal if $V_\kappa$ is a model of $\ZF_2$, the second-order version of $\ZF$. This is equivalent to the definition ``For all $x\in V_\kappa$, any $f\colon x\to\kappa$ has a bounded image''.

Much as is the case in $\ZFC$, we say that $\kappa$ is a Mahlo cardinal if it is an inaccessible such that $\{\alpha<\kappa\mid\alpha\text{ is a regular cardinal}\}$ is a stationary set.

We say that a cardinal $\kappa>\omega$ is a measurable cardinal if it carries a $\kappa$-complete non-principal ultrafilter.\footnote{There is a case to be argued in favour of allowing $\omega$ to be measurable, especially in the choiceless context. In this work, however, it is simpler to stick with the traditional $\ZFC$-style definition which excludes $\omega$.} We can prove that if $\kappa$ is measurable, then it is a regular cardinal and that if $\lambda<\kappa$, then $\kappa\nleq2^\lambda$, namely there is no injection from $\kappa$ to $\power(\lambda)$. However we cannot prove that there is no surjection from $2^\lambda$ onto $\kappa$, e.g.\ $\omega_1$ can be measurable (see \cite{Jech:1968}, for example).

For infinite cardinals $\kappa,\lambda$ we denote by $S^\kappa_\lambda=\{\alpha<\kappa\mid\cf(\alpha)=\lambda\}$. We say that a stationary subset $S\subseteq S^\kappa_\omega$ \textit{reflects} if there is some $\alpha<\kappa$ such that $\cf(\alpha)>\omega$ and $S\cap\alpha$ is stationary in $\alpha$. If no such $\alpha$ exists, then $S$ is a non-reflecting stationary set. It is a standard exercise to see that if $\kappa$ is weakly compact, then every stationary subset of $S^\kappa_\omega$ reflects.

If $\kappa$ is a measurable cardinal and $U,U'$ are two normal measures on $\kappa$, we write $U\lhd U'$ if $U$ belongs to the ultrapower of $V$ by $U'$.\footnote{This definition makes sense even without assuming the Axiom of Choice.} This is a well-founded relation, and we write $o(\kappa)$, the Mitchell order of $\kappa$, to denote the range of its rank function.\footnote{See \cite[Definition~19.33]{Jech2003} for details.} Note that $o(\kappa)=0$ if $\kappa$ is not measurable, $o(\kappa)=1$ if and only if $\kappa$ is measurable, but is not measurable in any ultrapower by its measures, and if $o(\kappa)\geq 2$, then there is a normal measure which contains the set $\{\alpha<\kappa\mid\alpha\text{ is measurable}\}$. In particular, higher Mitchell order implies that there are many measurable cardinals in the universe.

Throughout the course of the paper we will be interested in elementary embeddings between transitive sets in a choiceless context. Let us stress that the definition of an elementary embedding between \emph{set structures} $j \colon \mathcal{A} \to \mathcal{B}$ makes perfect sense without the Axiom of Choice. Indeed, the deep connection between elementary embeddings and ultrafilters (via {\L}o\'{s}'s theorem), which holds in $\ZFC$, can be avoided when discussing elementary embeddings between set structures in a choiceless context.
\subsection{Symmetric extensions}
Let $\PP$ be a forcing notion. If $\pi\in\aut(\PP)$ then $\pi$ extends its action to $\PP$-names with the following recursive definition, \[\pi\dot x=\{\tup{\pi p,\pi\dot y}\mid\tup{p,\dot y}\in\dot x\}.\]
Let $\sG$ be a group. We say that $\sF$ is a filter of subgroups of $\sG$ if it is a non-empty family of subsgroups of $\sG$ closed under finite intersections and supergroups. We say that $\sF$ is a \textit{normal filter}\footnote{To avoid confusion with normal ultrafilters in the context of large cardinals, we will refer to those as normal \textit{measures}.} of subgroups if additionally for every $\pi\in\sG$ and $H\in\sF$, $\pi H\pi^{-1}\in\sF$.

\begin{definition}
  A \textit{symmetric system} is of the form $\tup{\PP,\sG,\sF}$ where $\PP$ is a forcing notion, $\sG\subseteq\aut(\PP)$ is a group of automorphisms, and $\sF$ is a normal filter of subgroups on $\sG$.
\end{definition}

Let $\tup{\PP,\sG,\sF}$ be a symmetric system and let $\dot x$ be a $\PP$-name. We say that $\dot x$ is \textit{$\sF$-symmetric} if $\sym_\sG(\dot x)=\{\pi\in\sG\mid \pi\dot x=\dot x\}\in\sF$. If this property holds hereditarily to all names which appear in $\dot x$ as well, we will say that $\dot x$ is \textit{hereditarily $\sF$-symmetric}. We use $\HS_\sF$ to denote the class of all hereditarily $\sF$-symmetric names, although we will omit the subscripts where no confusion can occur.

The following lemma and theorem can be found in \cite{Jech2003} as Lemma~14.37 and Lemma~15.51 respectively.

\begin{lemma}[The Symmetry Lemma]
Suppose that $p\in\PP$, $\dot x$ is a $\PP$-name, and $\pi\in\aut(\PP)$. Then $p\forces\varphi(\dot x)\iff\pi p\forces\varphi(\pi\dot x).$
\end{lemma}

\begin{theorem}
  Suppose that $\tup{\PP,\sG,\sF}$ is a symmetric system and $G\subseteq\PP$ is a $V$-generic filter. Then $\HS^G=\{\dot x^G\mid\dot x\in\HS\}$ is a transitive subclass of $V[G]$ which contains $V$ and satisfies all the axioms of $\ZF$.
\end{theorem}
We refer to $\HS^G$ as a \textit{symmetric extension} of $V$. We can relativise the $\forces$ relation to the class $\HS$ and obtain $\forces^\HS$, which allows us to discuss truth values of statements in $\HS^G$ in the ground model. The symmetric forcing relation satisfies the Forcing Theorem as well as the Symmetry Lemma, if we restrict $\pi$ to $\sG$.

While we will not expand on the subject of iterating symmetric extensions in this paper, we will refer to two concepts that arise from the work of the second author in \cite{Karagila2016}.

\begin{definition}
Let $\tup{\PP,\sG,\sF}$ be a symmetric system. We say that $D\subseteq\PP$ is \textit{symmetrically dense} if for some $H\in\sF$, every $\pi\in H$ satisfies $\pi``D=D$. We say that $G$ is  \textit{symmetrically $V$-generic} if for every symmetrically dense $D\in V$, $D\cap G\neq\varnothing$.
\end{definition}

\begin{theorem}[Theorem~8.4 in \cite{Karagila2016}]
  Let $\tup{\PP,\sG,\sF}$ be a symmetric system and $\dot x\in\HS$. The following are equivalent:
  \begin{enumerate}
  \item $p\forces^\HS\varphi(\dot x)$.
  \item For every symmetrically generic filter $G$ with $p\in G$, $\HS^G\models\varphi(\dot x^G)$.
  \item For every generic filter $G$ with $p\in G$, $\HS^G\models\varphi(\dot x^G)$.
  \end{enumerate}
\end{theorem}

\begin{definition}
  Let $\PP\ast\dot\QQ$ be a two-step iterated forcing, let $\sG$ be a group of automorphisms of $\PP$ such that $\forces_\PP\pi\dot\QQ=\dot\QQ$ for all $\pi\in\sG$, and let $\dot\sH$ be a name for an automorphism group of $\dot\QQ$. The \textit{generic semi-direct product}, $\sG\ast\dot\sH$, is the group of automorphisms of $\PP\ast\dot\QQ$ which can be represented as $\tup{\pi,\dot\sigma}$ with $\pi\in\sG$ and $\forces_\PP\dot\sigma\in\dot\sH$, and whose action on $\PP\ast\dot\QQ$ is given by \[\tup{\pi,\dot\sigma}(p,\dot q)=\tup{\pi p,\pi(\dot\sigma\dot q)}.\]
\end{definition}
This is a natural notion which is the first step towards understanding how to iterate symmetric extensions as a whole, and we will normally have a symmetric system and require that $\dot\QQ$ and $\dot\sH$ are in $\HS$ altogether. We remark that in the case of a product of two symmetric systems, since $\sH$ is in the ground model, it follows that $\sG\ast\check\sH=\sG\times\sH$ acting pointwise on $\PP\times\QQ$.

\subsection{Large cardinal axioms and preservation thereof}
Suppose that $G$ is a $V$-generic filter for some forcing and $\kappa$ is a cardinal in $V$. We say that $A \in \power^{V[G]}(\kappa)$ is \textit{essentially small with respect to $\kappa$} if $V[A]$ is a generic extension by a forcing of size $<\kappa$. When $\kappa$ is clear from the context we will omit it. The following theorem is essentially due to Jech \cite{Jech:1968}.
\begin{theorem}\label{thm:jech-measurable}
  Let $V\subseteq W\subseteq V[G]$ be models of $\ZF$ with $G$ a $V$-generic filter, and let be $\kappa$ a measurable cardinal in $V$. If every $A\in\power(\kappa)^W$ is essentially small, then every measure on $\kappa$ in $V$ extends to a measure in $W$.
\end{theorem}

\begin{definition}\label{def:weakly-critical}
We say that $\kappa$ is a \textit{weakly critical cardinal} if for every $A\subseteq V_\kappa$ there exists an elementary embedding $j\colon X\to M$ with critical point $\kappa$, where $X$ and $M$ are transitive and $\kappa,V_\kappa,A\in X\cap M$.
\end{definition}
\begin{proposition}[Proposition~3.2 in \cite{HayutKaragila:Critical}]
$\kappa$ is weakly critical if and only if for every $A\subseteq V_\kappa$ there is a transitive, elementary end-extension of $\tup{V_\kappa,\in,A}$.
\end{proposition}

\autoref{def:weakly-critical}, as studied in-depth in \cite{HayutKaragila:Critical}, is the embedding-based definition of weakly compact cardinals in a choiceless context. This definition is quite strong and implies inaccessibility and Mahloness, as well as stationary reflection.\footnote{In the literature, the notion of weak compactness in a choiceless context typically refers to the partition property $\kappa\to(\kappa)^2_2$ (see, e.g., \cite{Apter:1983,Bull1978}), which is quite weak without the Axiom of Choice. Properties which are equivalent under the Axiom of Choice to weak compactness (such as the tree property with inaccessibility) have several non-equivalent formulations in a choiceless context.}

\begin{corollary}\label{cor:weakly-critical-implies-stationary-reflection}
If $\kappa$ is weakly critical, then every stationary $S\subseteq\kappa$ reflects.
\end{corollary}
\begin{proof}
  Let $\tup{V_\kappa,\in,S}\prec M$ be an elementary end-extension of $V_\kappa$. Then $M\models``S\cap\kappa$ is stationary'', therefore $M\models``S$ reflects'', so $V_\kappa$ models that as well.
\end{proof}

\begin{definition}
We say that a cardinal $\kappa$ is a \textit{critical cardinal} if it is the critical point of an elementary embedding $j\colon V_{\kappa+1}\to M$, where $M$ is a transitive set. \end{definition}

Let $\cS=\tup{\PP,\sG,\sF}$ be a symmetric system and let $j\colon V\to M$ be an elementary embedding. In \cite{HayutKaragila:Critical} we identified a sufficient condition for the embedding $j$ to be amenably lifted to the symmetric extension given by $\cS$. Namely, if $W$ is the symmetric extension of $V$ given by $\cS$, and $N$ is the symmetric extension of $M$ given by $j(\cS)$, the following condition guarantees that $j\restriction W_\alpha\in W$ for all $\alpha$.

\begin{definition}\label{def:j-decomposable}
Let $\cS=\tup{\PP,\sG,\sF}$ be a symmetric system and let $j\colon V\to M$ be an elementary embedding. We say that $\cS$ is \textit{$j$-decomposable} if the following conditions hold.\footnote{This means that $j(\cS)$ is essentially a two-step iteration of symmetric systems with an $M$-generic for the second iterand.}
\begin{enumerate}
\item There is a condition $m\in j(\PP)$ and a name $\dot\QQ\in\HS_\sF$ such that:
\begin{enumerate}
\item $\pi\colon j(\PP)\restriction m\cong\PP\ast\dot\QQ$, $\pi$ extends the function $j(p)\mapsto\tup{p,1_\QQ}$,
\item $\sym(\dot\QQ)=\sG$,
\item there is $\dot H\in\HS_\sF$ such that $\forces_\PP\dot H\text{ is symmetrically }\check M\text{-generic for }\dot\QQ$.
\end{enumerate}
\item There is a name $\dot\sH$ such that:
\begin{enumerate}
\item $\dot\sH\in\HS_\sF$ and $\forces_\PP\dot\sH\leq\aut(\dot\QQ)$, with $\sym(\dot\sH)=\sG$,
\item there is an embedding $\tau\colon\sG\to\sG\ast\dot\sH$, the generic semi-direct product, such that $\tau(\sigma)$ is given by applying $\pi$ to $j(\sigma)$,
\item $\tau(\sigma)=\tup{\sigma,\dot\rho}$, and $\dot\rho$ is a name such that $\forces_\PP\dot\rho``\dot H=\dot H$.
\end{enumerate}
\item The family $j``\sF$ is a basis for $j(\sF)$.
\end{enumerate}
\end{definition}

The following is Theorem~4.7 in \cite{HayutKaragila:Critical}.
\begin{theorem}[The Basic Lifting Theorem]\label{thm:lifting-criteria}
If $j\colon V\to M$ is an elementary embedding and $\cS$ is a $j$-decomposable symmetric system, then $j$ can be amenably lifted to the symmetric extension defined by $\cS$.
\end{theorem}

By analysing the proof, it turns out that we can weaken condition (3) in the definition of $j$-decomposable to the following:
\begin{enumerate}
\item[(3')] For every $K\in j(\sF)$, there is some $K_0\in\sF$ such that $j``K_0\subseteq K$.
\end{enumerate}
So instead of requiring that $j(K_0)\subseteq K$, we simply require that $j``K_0\subseteq K$. This is enough to let the proof of Corollary~4.5 in \cite{HayutKaragila:Critical} work. This provides us with an improved lifting theorem for elementary embeddings.

In both of our applications below, $j(\PP)\cong \PP\times\RR$ for some reasonably well-behaved forcing $\RR$ such that we can find $M$-generic filters for $\RR$ in $V$. In particular, $\dot{\mathbb{Q}}$ in the definition is the canonical name of $\RR$, so it is fixed under any automorphism of $\mathbb{P}$. Moreover, under the isomorphism from $j(\PP)$ to $\PP\times\RR$, $j(p)$ is mapped to $\langle p, 1\rangle$, so (1)(a) holds for $m=1_{j(\PP)}$. This guarantees that (1) will be satisfied.

As these two applications are products of symmetric extensions, for any $\pi\in\sG$, $j(\pi)$ is of the form $\langle \pi, \id_{\RR}\rangle$. This guarantees that (2) will always be satisfied.
\section{Small measurable cardinals}
The first two results will have a similar proof. We begin with a measurable cardinal $\kappa$ and take a symmetric system which preserves measurability using \autoref{thm:jech-measurable}. The two systems we will use are nearly the same, with the first one designed to preserve weakly compact embeddings and therefore $\kappa$ will remain weakly critical. We will always assume $\ZFC+\GCH$ in the ground model so that we do not have to worry about collapsing any cardinals below $\kappa$.
\subsection{The first weakly critical cardinal}
\begin{theorem}\label{thm:weakly-critical}
  Let $\kappa$ be a measurable cardinal. There is a symmetric extension in which $\kappa$ is the least weakly critical cardinal and the least measurable cardinal.
\end{theorem}
In order to kill the weak compactness of cardinals below $\kappa$, we will use the standard forcing for adding a non-reflecting stationary set. Given a regular cardinal $\alpha$ we define a forcing $\QQ_\alpha$ which adds a non-reflecting stationary subset of $S^\alpha_\omega$. A condition $p\in\QQ_\alpha$ is a bounded subset of $S^\alpha_\omega$ such that if $\beta<\sup p$ and $\cf(\beta)>\omega$, then there is a club $c\subseteq\beta$ such that $c\cap p=\varnothing$. We say that $q\leq p$ if $q$ is an end-extension of $p$.

If $S$ is the generic subset added by $\QQ_\alpha$, we claim that $S$ is stationary (it is easy to see that $S$ is non-reflecting). Suppose that $p\forces\dot C$ is a club. We define by recursion a decreasing sequence of conditions, $p_{n+1}\leq p_n\leq p$, such that there is some $\alpha_n>\sup p_n$ for which $p_{n+1}\forces\check\alpha_n\in\dot C$. Let $\alpha=\sup_{n<\omega}\alpha_n$ and let $q=\bigcup_{n<\omega}p_n\cup\{\alpha\}$. Note that if $\beta\leq\alpha$ and $\cf(\beta)>\omega$ then there is some $n$ such that $\beta<\alpha_n$, and therefore $q\cap\beta=p_n\cap\beta$, therefore $q$ is a viable condition in $\QQ_\alpha$. Moreover, $q\forces\check\alpha\in\dot C\cap\dot S$, where $\dot S$ is the canonical name for $S$. So any condition must force that $S$ meets any club, and therefore it is stationary.

We refer the reader to \S6.5 of \cite{Cummings:Handbook} where it is shown that this forcing, while not $\kappa$-closed, is $\kappa$-strategically closed. We also claim that $\QQ_\alpha$ is weakly homogeneous. This is a very simple consequence of the following characterisation of weak homogeneity.

\begin{theorem}\label{thm:homogeneity}
  The following are equivalent for a forcing $\PP$:
  \begin{enumerate}
  \item The Boolean completion of $\PP$ is weakly homogeneous.
  \item If $G$ is a $V$-generic filter, then for every $p\in\PP$ there is a $V$-generic filter $H$ over $\PP$, such that $p\in H$ and $V[G]=V[H]$.
  \end{enumerate}
\end{theorem}

This is a consequence of a theorem of Vop\v{e}nka and H\'ajek. See Theorem~1 in Section~3.5 of \cite{Grigorieff:1975} and Theorem~8 in \cite{HayutKaragila:Restrictions}.

Let $V$ be the ground model and let $\PP\in V$ be the Easton support product of $\QQ_\alpha$, where $\alpha<\kappa$ is an inaccessible cardinal. Our group of automorphisms is the Easton support product of $\aut(\QQ_\alpha)$ (acting pointwise), and our filter of subgroups is generated by subgroups of the form $\prod_{\alpha<\kappa}H_\alpha$, where $H_\alpha$ is a subgroup of $\QQ_\alpha$ for all $\alpha$, and $H_\alpha=\aut(\QQ_\alpha)$ for all but finitely many coordinates.\footnote{We can replace ``finitely'' by any fixed $\lambda<\kappa$.} We let $W$ denote the symmetric extension given by a $V$-generic filter $G$.
\begin{proposition}\label{prop:ess-small-easton}
Suppose that $\dot A\in\HS$ and $p\forces\dot A\subseteq\check\kappa$. Then there is some $\alpha<\kappa$ and $\dot{B}$ such that $p \forces \dot A = \dot B$ and $\dot B$ is a $\PP\restriction\alpha$-name. In particular, every subset of $\kappa$ in $W$ is essentially small, so $\kappa$ remains measurable.
\end{proposition}
\begin{proof}
  This is a standard homogeneity argument. Let $\alpha$ be such that $\sym(\dot A)$ contains $\prod_{\beta<\alpha}\{\id\}\times\prod_{\beta\geq\alpha}\aut(\QQ_\beta)$. If $p\forces\check\xi\in\dot A$, and $q\restriction \alpha \leq p\restriction\alpha$, then there is an automorphism $\pi\in\sym(\dot A)$ such that $\pi p$ is compatible with $q$. This implies that $p\restriction\alpha\forces\check\xi\in\dot A$, and the same holds, of course if $p\forces\check\xi\notin\dot A$. Therefore $\{\tup{p\restriction\alpha,\check\xi}\mid p\forces\check\xi\in\dot A\}$ is a $\PP\restriction\alpha$-name equivalent to $\dot A$.

  For the ``In particular'' part, note that $|\mathbb{P}\restriction \alpha| \leq 2^\alpha < \kappa$.
\end{proof}

It is routine to verify that for every ordinal $\alpha$, $\cf^{V[G]}(\alpha) = \cf^V(\alpha)$, and in particular, every cardinal in $V$ remains a cardinal in $V[G]$. By \autoref{prop:ess-small-easton}, and standard downwards and upwards absoluteness arguments, for every ordinal $\alpha$, $\cf^W(\alpha) = \cf^V(\alpha)$, and the cardinals of $W$ are the same as the cardinals of $V$.

Let us remark that $W_\kappa$ is much smaller than $(V[G])_\kappa$, as a subset of ordinals that can enter $W_\kappa$ must be introduced by a product of finitely many components $\mathbb{Q}_\alpha$. In particular, $W_\kappa\models \neg \AC_\omega$.

It is immediate that every $\alpha<\kappa$ which is inaccessible in $W$, and thus in $V$, has a non-reflecting stationary subset and cannot be weakly critical, by \autoref{cor:weakly-critical-implies-stationary-reflection}. As weakly critical cardinals are inaccessible, the following propositions complete the proof of \autoref{thm:weakly-critical}.
\begin{proposition}\label{prop:no-measurable-below}
There are no measurable cardinals below $\kappa$ in $W$.
\end{proposition}
\begin{proof}
Let $\alpha < \kappa$ be a regular cardinal in the symmetric extension. So, in particular, $\alpha$ is regular in $V$. Let $S\subseteq \alpha$ be the generic non-reflecting stationary set for $\QQ_\alpha$, if $\alpha$ is inaccessible in $V$, and the empty set otherwise. Let us assume, towards a contradiction, that there is a measure $\cU_\alpha$ on $\alpha$ in the symmetric extension. Therefore, in the symmetric extension, one can form the ultrapower $\Ult(V[S], \cU_\alpha)$. As $V[S]$ is a model of $\ZFC$ (even though $\cU_\alpha$ might be external to $V[S]$), we can apply {\L}o\'{s}'s theorem and obtain an elementary embedding $j \colon V[S] \to M$ with critical point $\alpha$.

Let us deal first with the case that $\alpha$ is inaccessible in $V$ and thus $S$ is a generic non-reflecting stationary subset of $\alpha$.

We have that $M=N[j(S)]$ where $N=j(V)$ and $j(S)$ is $N$-generic for $j(\QQ_\alpha)$. In particular, %since $\alpha$ is a regular cardinal
there is some $q\in j(\QQ_\alpha)$ such that $S\subseteq q\subseteq j(S)$ and thus there is a club $C\subseteq\alpha$ for which $C\cap q=\varnothing$, $C\in N$. But $N\subseteq W$ so it must be that $C$ is a club in $W$ and $C\cap S=\varnothing$. By \autoref{prop:ess-small-easton}, there is some $\beta$ for which $C$ is added by $\PP\restriction\beta$.

Since $S$ is stationary in $V[S]$, as it is $V$-generic for $\QQ_\alpha$, and $\PP\restriction\alpha$ is $\alpha$-c.c., it must be that $C$ was added by $\PP\restriction[\alpha,\kappa)$. However, as $\PP\restriction(\alpha,\kappa)$ is $\alpha^+$-distributive, it has to be the case that $C$ is added by $\QQ_\alpha$ itself, so $C\in V[S]$, which is a contradiction to the stationarity of $S$.

By $\GCH$, the other possible case is that $\alpha$ was a successor cardinal in $V$. So, as $\alpha$ is the critical point of $j$, $\alpha$ is not a cardinal in $N$. But $N \subseteq W$ and $\alpha$ is a cardinal in $W$ -- a contradiction.
\end{proof}
\begin{proposition}
  $\forces^\HS\check\kappa$ is weakly critical.
\end{proposition}
\begin{proof}
  Work in $V$. Let $j\colon M\to N$ be an elementary embedding, where $M$ and $N$ are $\kappa$-models (so $M, N, j \in V$). Then $j(\PP)=\PP\times\RR$, with $\RR$ being $\kappa$-strategically closed. Using the fact that $N$ is a $\kappa$-model, we can construct an $N$-generic filter for $\RR$ in $V$. All the properties of \autoref{def:j-decomposable} hold for $\PP$, except for (3), since $j(\PP)$ decomposes to a product. We will show that (3') holds, which allows us to apply \autoref{thm:lifting-criteria}.

  Suppose that $B\in j(\sF)$. We write $B_\alpha$ as the projection of $B$ on the $\alpha$th group in the product. Let $E=\{\alpha<j(\kappa)\mid B_\alpha\neq\aut(\QQ_\alpha)\}$. Then $E$ is finite. We have that $B\restriction\kappa=\{\pi\restriction\kappa\mid\pi\in B\}\in\sF$. Therefore, we have that $j``B\restriction\kappa\subseteq B$, even though $j(B\restriction\kappa)$ may be strictly larger than $B$ itself.\footnote{For example, if $B$ is such that $\min E>\kappa$, then $B\restriction\kappa=\sG$ and $B\subsetneq j(B\restriction\kappa)=j(\sG)$.} So, we can lift $j$ to the symmetric extensions of $M$ and $N$ as needed.

  Let $\dot A\in\HS$ be a name for a predicate on $W_\kappa=V_\kappa^W$. Let $M\in V$ be a $\kappa$-model such that $\PP,\sG,\sF,\dot A,\dot W_\kappa\in M$ and $j\colon M\to N$ is an elementary embedding of $\kappa$-models in $V$. Then $j$ lifts to the symmetric extension, and it is clear that the symmetric extension of $M$ contains $\dot A^G$, so $\kappa$ remains weakly critical in $W$.
\end{proof}
\subsection{The first Mahlo cardinal}
\begin{theorem}\label{thm:mahlo}
  Let $\kappa$ be a measurable cardinal. There is a symmetric extension in which $\kappa$ is the least measurable cardinal and the least Mahlo cardinal.
\end{theorem}
We repeat the same proof as before. This time $\QQ_\alpha$ is the standard club shooting forcing to shoot a club through the singulars at every inaccessible cardinal $\alpha<\kappa$. Namely, the conditions are closed and bounded sets of singular ordinals below $\alpha$, ordered by end-extension. By \autoref{thm:homogeneity}, this forcing is weakly homogeneous. The group and filter are defined as before as the Easton and finite support products, respectively. We use $W$ to denote the symmetric extension, as before. It is easy to see that if $\alpha<\kappa$, then $\alpha$ is not Mahlo in $W$ (being weakly Mahlo is downwards absolute from $W$ to $V$, and by $\GCH$ in $V$, weak Mahloness implies Mahloness).

The following closure property of $\mathbb{Q}_\alpha$ is well-known.
\begin{remark}
For every regular cardinal $\gamma < \alpha$, there is a dense subset of $\mathbb{Q}_\alpha$ which is $\gamma$-closed.
\end{remark}
From this closure property, it follows that under $\GCH$ cardinals are not collapsed: Indeed, if a cardinal $\alpha < \kappa$ is collapsed to cardinality $\gamma$, then the witness for the collapse must be added by the iteration up to and including $\gamma$. But the cardinality of this iteration is bounded by $\gamma$. Without $\GCH$ we still have that no inaccessible cardinal in the ground model was collapsed.

The following proposition is proved similarly to \autoref{prop:ess-small-easton}, as the proof only uses the homogeneity of the components of the forcing.

\begin{proposition}
  Every subset of $\kappa$ in $W$ is essentially small with respect to $\kappa$. In particular, $\kappa$ remains measurable in $W$.\qed
\end{proposition}

\begin{proposition}
  $\kappa$ remains a Mahlo cardinal in $W$.
\end{proposition}
\begin{proof}
  Let $A\subseteq\kappa$. Since $A$ is essentially small, in $V[A]$, $\kappa$ remains Mahlo. So if $A$ is a club, in $V[A]$, $A$ contains an inaccessible cardinal. Since $\PP$ preserves cofinalities and $\GCH$, $A$ contains an inaccessible cardinal in the generic extension by $\mathbb{P}$ and by downwards absoluteness -- in $W$ as well.
\end{proof}
The following proposition completes the proof of \autoref{thm:mahlo} by showing that no measurable cardinals below $\kappa$ in $W$. It is proved in exactly the same way as \autoref{prop:no-measurable-below}.
\begin{proposition}
  There are no measurable cardinals below $\kappa$ in $W$.\qed
\end{proposition}
\subsection{The first inaccessible?}
\begin{theorem}\label{thm:failure-of-covering}
  Suppose that the first measurable cardinal is the first inaccessible cardinal. Then there is an inner model with a measurable cardinal of Mitchell order $2$. In particular, it is impossible to start with a single measurable cardinal and construct a model in which the first measurable cardinal is the first inaccessible cardinal.
\end{theorem}
\begin{proof}
  Towards a contradiction, assume that there is no inner model satisfying $\exists\kappa(o(\kappa)\geq2)$. Let $\kappa$ be the least inaccessible cardinal, which is also the least measurable cardinal. Fix a measure $U$. We do not necessarily assume that $U$ is normal, as there might be no normal measures on $\kappa$ (see \cite{BilinskyGitik2012}).

  Since $\kappa$ is not Mahlo, let $C\subseteq\kappa$ be a club through the singulars. Since every set of ordinals is generic over $\HOD$, the Mitchell core model, $K$, is unchanged between $\HOD$ and $\HOD[C]$. Moreover, we can encode $U\cap\HOD[C]$ as a set of ordinals and by forcing further add it to the universe as well without changing $K$. Let $M$ denote $\HOD[C][U\cap\HOD[C]]$. Then, as we remarked, $K^\HOD=K^M$.

  In $M$, use $U\cap\HOD[C]$ to define an elementary embedding $j\colon\HOD[C]\to W$, where $W$ is some transitive class, and derive a normal $\HOD[C]$-measure, $D$, from this embedding. Moreover, $D\cap K$ is a normal $K$-measure. By the maximality of $K$, $D\cap K\in K$, as shown by Mitchell in \cite{Mitchell:1984}.

  Since $\kappa$ is measurable in $K$ and $K\models o(\kappa)<2$, we have \[\{\alpha<\kappa\mid K\models\alpha\text{ is non-measurable}\}\in D\cap K.\] But since $D$ is a normal measure over $\HOD[C]$ on $\kappa$ and $C$ is a club, that means that $C$ cannot contain any $\alpha$ which is measurable in $K$.

  Finally, as $D \cap K$ is a normal measure, the set $R = \{\alpha < \kappa \mid \cf^{K} (\alpha) = \alpha\} \in D$. So, pick any $\alpha\in C \cap R$. Then, we have a singularising sequence $A\subseteq\alpha$ in $V$, which is generic over $\HOD$. As $K$ remains the same between $\HOD$ and $\HOD[A]$, we have now a model of $\ZFC$ in which $\alpha$ is singular, but regular in $K$. Therefore, by Theorem~2.5 in \cite{Mitchell:CoveringLemmaHB}, $\alpha$ must be measurable in $K$, contrary to the assumption.
\end{proof}
\begin{remark}
  If we can prove that we can find $\alpha$ of uncountable cofinality in $V$ which is $K$-measurable, then we can actually push the consistency strength much higher than just $o(\kappa)\geq 2$.\footnote{Since the writing of this manuscript the authors, joint with Gitik, have shown that in the Mitchell core model $o(\kappa)\geq\kappa+1$, as well as established the consistency of the least measurable cardinal being the least inaccessible cardinal from a supercompact cardinal. See \cite{GHK:2024} for details.}
\end{remark}
\begin{question}
  What is the consistency strength of ``the least inaccessible cardinal is the least measurable cardinal''?
\end{question}
\begin{question}
  Suppose that $\kappa$ is the least inaccessible cardinal and the least measurable cardinal. Is it always the case that $\kappa$ carries a normal measure?
\end{question}
\bibliographystyle{amsplain}
\providecommand{\bysame}{\leavevmode\hbox to3em{\hrulefill}\thinspace}
\providecommand{\MR}{\relax\ifhmode\unskip\space\fi MR }
\providecommand{\MRhref}[2]{%
  \href{http://www.ams.org/mathscinet-getitem?mr=#1}{#2}
}
\providecommand{\href}[2]{#2}

\end{document}